\documentclass[reqno,12pt,   letterpaper]{amsart}
\usepackage{amsmath,amssymb,amsthm,graphicx,mathrsfs,url}
\usepackage[usenames,dvipsnames]{color}
\usepackage[colorlinks=true,linkcolor=Red,citecolor=Green]{hyperref}

\usepackage{mathtools}

\def\?[#1]{\textbf{[#1]}\marginpar{\Large{\textbf{??}}}}

\newtheorem{thm}{Theorem}[section]
\newtheorem{prop}{Proposition}[section]

\newtheorem{cor}[prop]{Corollary}

\numberwithin{equation}{section}

\DeclareMathOperator{\supp}{supp}

\renewcommand\Re{\mathrm{Re}\,  }

\newcommand{\R}{\mathbb{R}}

\newcommand{\C}{\mathbb{C}}

\newcommand\sskkip[1]{}

\renewcommand{\Re}{\hbox{Re}\,}
\renewcommand{\Im}{\hbox{Im}\,}

\newcommand{\p}{\partial}

 \title[Stability estimates  in the high frequency limit]{Stability estimates for partial data inverse problems for Schr\"odinger operators in the high frequency limit}

\author[Krupchyk]{Katya Krupchyk}

\address
        {K. Krupchyk, Department of Mathematics\\
University of California, Irvine\\ 
CA 92697-3875, USA }

\email{katya.krupchyk@uci.edu}

\author[Uhlmann]{Gunther Uhlmann}

\address
       {G. Uhlmann, Department of Mathematics\\
       University of Washington\\
       Seattle, WA  98195-4350\\
       USA\\
       and Institute for Advanced Study of the Hong Kong University of Science and Technology}
\email{gunther@math.washington.edu}

\begin{document}


\begin{abstract} 

We consider the partial data inverse boundary problem for the Schr\"odin\-ger operator at a frequency $k>0$ on a bounded domain in $\R^n$, $n\ge 3$, with impedance boundary conditions.  Assuming that the potential is known in a neighborhood of the boundary, we first show that the knowledge of the partial Robin--to--Dirichlet map  at the fixed frequency $k>0$ along an arbitrarily small portion of the boundary,  determines the potential in a logarithmically stable way. We prove, as the principal result of this work, that the logarithmic stability can be improved to the one of H\"older type in the high frequency regime.  

\end{abstract}

\maketitle

\textbf{R\'esum\'e.} Nous consid\'erons un probl\`eme inverse avec des donn\'ees partielles pour l'op\'erateur de Schr\"odinger \`a la fr\'equence $k>0$ sur un domaine  born\'e dans $\R^n$, $n\geq 3$, avec des conditions aux limites d'imp\'edance. En supposant que le potential soit connu sur un voisinage du bord, nous montrons d'abord que la connaissance de l'op\'erateur Robin-Dirichet partiel \`a la fr\'equence fix\'ee $k>0$ sur des sous-ensembles
arbitrairement petits du bord, d\'etermine le potentiel de mani\`ere logarithmiquement stable. Nous montrons, comme le r\'esultat principal de ce travail, que la stabilit\'e logarithmique peut \^etre am\'elior\'ee et remplac\'ee par une estimation de stabilit\'e de type H\"old\'erienne dans le r\'egime des hautes fr\'equences.

\section{Introduction and statement of results}
Let $\Omega\subset \R^n$, $n\ge 3$, be a bounded connected open set with $C^\infty$ boundary,  let $k>0$, and $q\in L^\infty(\Omega;\R)$. For any  $f\in H^{-\frac{1}{2}}(\p \Omega)$,  the  interior impedance problem 
\begin{equation}
\label{eq_int_1}
\begin{aligned}
&(-\Delta-k^2+q)u=0\quad \text{in}\quad \Omega,\\
&(\p_\nu-i k)u=f\quad \text{on}\quad \p \Omega,
\end{aligned}
\end{equation}
has a unique solution $u\in H^{1}(\Omega)$, see Proposition \ref{prop_solvability_direct}.  Here $\nu$ is the inner unit normal to $\p \Omega$.  Associated to the boundary problem \eqref{eq_int_1}, is the Robin--to--Dirichlet map defined by
\begin{equation}
\label{eq_int_Robin_to_Dir}
\Lambda_q(k): H^{-\frac{1}{2}}(\p \Omega)\to H^{\frac{1}{2}}(\p \Omega), \quad f\mapsto u|_{\p \Omega}.
\end{equation}
Thanks to Proposition \ref{prop_regularity_RtD}, we have 
\[
\Lambda_q(k): L^2(\p \Omega)\to H^{1}(\p\Omega).
\]
Let $\Gamma\subset \p \Omega$ be an arbitrary non-empty open subset of the boundary of $\Omega$, and let us define the corresponding partial Robin--to--Dirichlet map, 
\begin{equation}
\label{eq_int_mapping_prop}
\Lambda_q^\Gamma(k)=\tau_\Gamma\circ \Lambda_q(k) : L^2(\p \Omega)\to H^{1}(\Gamma),
\end{equation} 
where $\tau_\Gamma$ is the restriction map to $\Gamma$.  

The inverse boundary problem  with partial data that we are interested in is that of determining the potential $q$ in $\Omega$ from the knowledge of the partial Robin--to--Dirichlet map $\Lambda_q^\Gamma(k)$ at a fixed frequency $k> 0$.  
This problem has traditionally been studied  when the map $\Lambda_q^{\Gamma}(k)$ is replaced by the Dirichlet--to--Neumann map $u|_{\p \Omega}\mapsto \p_\nu u|_{\Gamma}$.  In the full data case when $\Gamma=\p \Omega$, the global uniqueness for this problem was established by Sylvester--Uhlmann in \cite{Syl_Uhl}.  Important progress on the partial data problem has been achieved by Ammari--Uhlmann  \cite{Ammari_Uhlmann}, Isakov \cite{Isakov_2007}, Bukhgeim--Uhlmann \cite{BukhUhl_2002}, Kenig--Sj\"ostrand--Uhlmann \cite{KenSjUhl2007},  Kenig--Salo \cite{Kenig_Salo_2013},  while the case when $\Gamma$ is arbitrary remains quite open, see \cite{Kenig_Salo_servey} for a review.  

Turning the attention to the stability aspects of this inverse problem, still in the case of the Dirichlet--to--Neumann map,  
a logarithmic stability estimate for the full data problem was established by Alessandrini  \cite{Aless_1}.  It has subsequently been shown by Mandache \cite{Mandache} that the logarithmic stability estimate is optimal.  In the case of partial data, logarithmic stability estimates complementing the uniqueness result of  Ammari--Uhlmann \cite{Ammari_Uhlmann}, were proved in \cite{Fathallah_2007},  see also \cite{Ben_Joud_2009},  \cite{Alessandrini_Kim_2012} and \cite{Ruland_Salo_2017}, while for the uniqueness result of Isakov \cite{Isakov_2007}, logarithmic stability estimates were obtained in  \cite{Heck_Wang_2016}. For the uniqueness results of Bukhgeim--Uhlmann \cite{BukhUhl_2002} and  Kenig--Sj\"ostrand--Uhlmann \cite{KenSjUhl2007},  the $\log-\log$ type stability estimates were established in \cite{Heck_Wang_2006}, and  in \cite{Caro_Dos_Santos_Ruiz}, \cite{Caro_Dos_Santos_Ruiz_2014}, respectively. 

The logarithmic stability estimates above indicate that the inverse problems considered are severely ill-posed, making it impossible to design reconstruction algorithms with high resolution in practice, since small errors in the boundary measurements will result in large errors in the reconstruction of the potential in the interior. Now it has been observed numerically that the stability may increase when the frequency $k$ of the problem becomes large \cite{Colton_Haddar}. The phenomenon of increasing stability in the large frequency regime, for several fundamental inverse problems, has been studied rigorously in \cite{Hry_Isakov}, \cite{Isakov_2007_inres} \cite{Isakov_2011}, \cite{Isakov_Wang_2014}, \cite{INUW_2014}, \cite{Isakov_Lai_Wang_2016}, among others. In particular, for the full data problem, assuming that the potential is sufficiently regular and is known near the boundary of $\Omega$, it was shown by Isakov   \cite{Isakov_2011} that  in the high frequency regime, the stability improves from a logarithmic one to the one of Lipschitz type, see also \cite{INUW_2014}. 

In the case of partial data inverse problems, the question of increasing stability at large frequencies has only been studied for the uniqueness result of Isakov \cite{Isakov_2007} in \cite{Heck_incres_stab_2017} and \cite{Liang_2015}, to the best of our knowledge. 

The goal of this paper is to study the issue of increasing stability in the uniqueness result of Ammari--Uhlmann \cite{Ammari_Uhlmann}, which establishes that under the assumption that the potential is known near the boundary, the knowledge of the  Dirichlet--to--Neumann map measured on an arbitrarily small portion $\Gamma$ of the boundary determines the potential in $\Omega$ uniquely.  In doing so we impose Robin boundary conditions, in order to have a unique solvability of the corresponding boundary problem \eqref{eq_int_1} for all $k>0$, whereas when working with Dirichlet boundary conditions, say, one has to assume that $k^2$ is not a Dirichlet eigenvalue of $-\Delta+q$. The latter requirement becomes impractical if one wishes to take the high frequency limit $k\to +\infty$, as we shall do in this paper. Let us also remark that the Robin boundary conditions are both natural and important, as they approximate Sommerfeld radiation conditions at high frequencies,  see \cite{Ammari_Dos_Santos_2013}, \cite{BaSpWu_2016}. They are also of interest to numerical analysts and play a fundamental role in the theory of integral equations for exterior problems, see \cite{BaSpWu_2016}.

Let us first state the following logarithmic stability estimate for the partial data inverse problem in the case of the Robin--to--Dirichlet map at a fixed frequency $k_0>0$. 
\begin{thm}
\label{thm_main_0}
Let $\Omega\subset \R^n$, $n\ge 3$, be a bounded connected open set with $C^\infty$ boundary, and  let $\Gamma\subset \p \Omega$ be an arbitrary non-empty open subset of the boundary of $\Omega$. Let $k_0> 0$ be fixed, let $M>0$ and let  $q_1,q_2\in L^\infty(\Omega;\R)$ be such that  $\|q_j\|_{L^\infty(\Omega)}\le M$.  Assume that  $q_1=q_2$ in  $\omega_0$, where  $\omega_0\subset \Omega$ is  a  neighborhood of  $\p \Omega$. Then there exists $C>0$ such that for $0< \delta:=\| \Lambda_{q_1}^\Gamma(k_0)-\Lambda_{q_2}^\Gamma(k_0)\|_{L^2(\p \Omega)\to H^{1}(\Gamma)}<1/e$, we have
\[
\|q_1-q_2\|_{H^{-1}(\Omega)}\le C  |\log \delta|^{-\frac{2}{n+2}}.
\]
Here $C>0$  depends on $k_0$,  $\Omega$, $\omega_0$, and  $M$. 
\end{thm}

The following is the main result of this paper, showing that the stability of the partial data inverse boundary problem increases as $k$ becomes large.   
\begin{thm}
\label{thm_main} 
Let $\Omega\subset \R^n$, $n\ge 3$, be a bounded connected open set with $C^\infty$ boundary, and  let $\Gamma\subset \p \Omega$ be an arbitrary non-empty open subset of the boundary of $\Omega$. Let $M>0$ and let  $q_1,q_2\in L^\infty(\Omega;\R)$ be such that  $\|q_j\|_{L^\infty(\Omega)}\le M$. Assume that  $q_1=q_2$ in  $\omega_0$, where  $\omega_0\subset \Omega$ is  a  neighborhood of  $\p \Omega$.   Then there exists a constant $C>0$  such that for all $k\ge 1$ and $0< \delta:=\| \Lambda_{q_1}^\Gamma(k)-\Lambda_{q_2}^\Gamma(k)\|_{L^2(\p \Omega)\to H^{1}(\Gamma)}<1/e$, we have 
\begin{equation}
\label{eq_thm_main}
\|q_1-q_2\|_{H^{-1}(\Omega)}\le e^{Ck}\delta^{\frac{1}{2}}+\frac{C}{(k+\log\frac{1}{\delta})^{\frac{2}{n+2}}}.
\end{equation}
Here $C>0$  depends on $\Omega$, $\omega_0$, $M$ but independent of $k$.
\end{thm}

\textbf{Remark 1}. The result of Theorem \ref{thm_main} can be summarized informally by saying that at high frequencies, the stability estimate is of H\"older type, modulo an error term with a power-like decay as $k$ becomes large.  

\textbf{Remark 2}. 
As is often the case for partial data inverse boundary problems, a fundamental role in the proof of Theorem \ref{thm_main} is played by Carleman estimates, here exploited to suppress the contribution of the boundary region away from $\Gamma$. The price to pay for relying on such exponentially weighted estimates is that the constant in the H\"older estimate \eqref{eq_thm_main}  grows exponentially in $k$.  It is an interesting problem to find conditions on the boundary portion $\Gamma$, in the spirit of control theory, see \cite{Burq_Zworski_2004},   that would allow one to replace the exponentially growing factor in \eqref{eq_thm_main} by a  polynomially growing one, as in the full data result  \cite{INUW_2014}.  

\textbf{Remark 3}. 
Theorem \ref{thm_main_0} follows from Theorem \ref{thm_main} and we shall therefore only be concerned with the proof of Theorem \ref{thm_main}.

Assuming that the potentials $q_1$ and $q_2$ enjoy regularity properties and a priori bounds that are better than $L^\infty$, we get the following corollary of Theorem \ref{thm_main}.  
\begin{cor}
\label{cor_main_int} 
Let $\Omega\subset \R^n$, $n\ge 3$, be a bounded connected open set with $C^\infty$ boundary, and  let $\Gamma\subset \p \Omega$ be an arbitrary non-empty open subset of the boundary of $\Omega$. Let $M>0$, $s>\frac{n}{2}$,  and let  $q_1,q_2\in H^s(\Omega;\R)$  be such that  $\|q_j\|_{H^s(\Omega)}\le M$. Assume that  $q_1=q_2$ in  $\omega_0$, where  $\omega_0\subset \Omega$ is  a  neighborhood of  $\p \Omega$.   Then there exists a constant $C>0$  such that for all $k\ge 1$ and $0< \delta:=\| \Lambda_{q_1}^\Gamma(k)-\Lambda_{q_2}^\Gamma(k)\|_{L^2(\p \Omega)\to H^{1}(\Gamma)}<1/e$, we have 
\[
\|q_1-q_2\|_{L^\infty(\Omega)}\le  \bigg( e^{Ck}\delta^{\frac{1}{2}}+\frac{C}{(k+\log\frac{1}{\delta})^{\frac{2}{n+2}}}\bigg)^{\frac{s-\frac{n}{2}}{2(s+1)}}.
\]
Here $C>0$  depends on $\Omega$, $\omega_0$, $M$, s  but independent of $k$.
\end{cor}

Let us now proceed to describe the main ideas in the proof of Theorem \ref{thm_main} along with the plan of the paper. The starting point is boundary Carleman estimates for the operator $-\Delta -k^2$, $k\ge 1$,  for functions $u$ satisfying the Robin boundary conditions $(\p_\nu  -ik )u=0$ on $\p \Omega$. Such Carleman estimates are essentially well known and are discussed in Section  \ref{sec_Carleman_estimates}, following the works by Fursikov--Imanuvilov \cite{Fursikov_Imanuvilov_1996},   Lebeau--Robbiano \cite{Lebeau_Robb_95}, \cite{Lebeau_Robb_97},  Burq \cite{Burq_2002}, and Buffe  \cite{Buffe_2017}, see also \cite{Le_Rousseau_2012} and \cite{Lebeau_book}. Let us mention that the presence of the large parameter $k$ in the boundary conditions makes the situation more complicated and in addition to $1/k$, it becomes natural to introduce a second small parameter $h$ such that $0 < h \ll 1/k$.  Using the boundary Carleman estimates and following the approach of \cite{Ben_Joud_2009}, in Section \ref{sec_consequences} we prove a  version of quantitative unique continuation from the boundary portion $\Gamma$,  valid for solutions of the Schr\"odinger equation, satisfying Robin boundary conditions. In Section \ref{sec_proof_thm_main}, using the estimates of Section \ref{sec_consequences} and a recent result by Baskin--Spence--Wunsch \cite{BaSpWu_2016} on bounds on solutions to the interior impedance problem, we obtain some crucial control on the difference of the potentials integrated against the product of solutions to the Schr\"odinger equations with a large frequency in terms of the difference of the corresponding partial Robin--to--Dirichlet maps. Taking solutions to be complex geometric optics solutions, we conclude the proof of Theorem \ref{thm_main} using standard arguments. Corollary \ref{cor_main_int} is established at the end of Section \ref{sec_proof_thm_main}. 
The paper is concluded by three appendices, assembled for the convenience of the reader. In  Appendix \ref{sec_direct_problem}, we discuss the solvability of the interior impedance problem \eqref{eq_int_1} and bounds on the solutions. Appendix \ref{app_carleman} is devoted to a sketch of the proof of Theorem \ref{thm_Fursikov_Imanuvilov}, due to Fursikov--Imanuvilov \cite{Fursikov_Imanuvilov_1996}.  Appendix \ref{app_CGO}
discusses complex geometric optics solutions for the Helmholtz equation with a potential following H\"ahner  \cite{Hahner_1996}.

\section{Semiclassical Carleman estimates with Robin boundary conditions}

\label{sec_Carleman_estimates}

Let $\Omega\subset \R^n$ be a bounded domain with $C^\infty$ boundary, and let $\nu$ be the unit inner normal to $\p \Omega$.  Let  
\[
P(h,E)=-h^2\Delta-E,
\]
where $0<h \le 1$ and  $0\le E\le 1$.  Letting $\varphi\in C^\infty(\overline{\Omega};\R)$,  we set  
\[
P_\varphi(h,E)=e^{\frac{\varphi}{h}}\circ P(h,E)\circ e^{-\frac{\varphi}{h}}.
\]
Our starting point is the following boundary Carleman estimates which are due to  Fursikov--Imanuvilov \cite{Fursikov_Imanuvilov_1996}, see also \cite{Le_Rousseau_2012}.  
\begin{thm}
\label{thm_Fursikov_Imanuvilov}
Let $\psi\in C^\infty(\overline{\Omega};\R)$ be such that $\psi\ge 0$ on  $\overline{\Omega}$ and $|\nabla \psi|>0$ in $\overline{\Omega}$,  and set  $\varphi=e^{\gamma \psi}$. Then there exist $\gamma_0>0$, $h_0>0$ and $C>0$, 
such that for all $\gamma\ge \gamma_0$,   $0<h\le h_0$, and $0\le E\le 1$, and  $u\in H^2 (\Omega)$, we have
\begin{equation}
\label{eq_thm_Furs_Iman}
\begin{aligned}
\int_{\Omega} |P_\varphi(h,E) u|^2 dx&+ h \bigg (\gamma^3\int_{\p \Omega} \varphi^3 |u|^2dS + \gamma \int_{\p \Omega} \varphi |h\nabla u|^2dS\bigg)\\
&\ge Ch
\big(  \gamma^4 \int_{\Omega} \varphi^3 |u|^2 dx +\gamma^2\int_{\Omega} \varphi |h\nabla u|^2dx \big).
\end{aligned}
\end{equation}
\end{thm}
Notice that in Theorem \ref{thm_Fursikov_Imanuvilov} there is no assumption that $\p_\nu\varphi|_{\p \Omega}\ne 0$, in contrast to  the boundary Carleman estimates of Burq \cite[Proposition 3.1]{Burq_2002}, see also  \cite{Lebeau_Robb_95}, \cite{Lebeau_Robb_97}.   For the convenience of the reader, a sketch of the  proof of Theorem \ref{thm_Fursikov_Imanuvilov} is given in  Appendix \ref{app_carleman}.

Let $\varphi\in C^\infty(\overline{\Omega}; \R)$. Next we shall discuss local boundary Carleman estimates  near a point $x_0\in \p \Omega$ where $\p_\nu \varphi(x_0)>0$ for functions satisfying Robin boundary conditions. These estimates can be obtained as a consequence of the local Carleman estimates for functions satisfying inhomogeneous Neumann boundary conditions which are due to Lebeau--Robbiano \cite[Proposition 2]{Lebeau_Robb_97}, or as a limiting case of  the local Carleman estimates for functions satisfying inhomogeneous Ventcel boundary conditions, due to Buffe  \cite[Theorem 1.5]{Buffe_2017}. To state the result, following \cite{Lebeau_Robb_95}, \cite{Lebeau_Robb_97}, \cite{Burq_2002}, and \cite{Buffe_2017},  we  require that $\varphi$ satisfies H\"ormander's hypoellipticity condition uniformly in $0\le E\le 1$: there exists $c>0$ such that for all $0\le E\le 1$, and  
\begin{equation}
\label{eq_105_2}
\text{all }(x,\xi)\in \overline{\Omega}\times \R^n,\quad p_\varphi(x,\xi,E)=0\quad \Longrightarrow\quad \{\Re p_\varphi, \Im p_\varphi \}\ge c,
\end{equation}
where $p_\varphi(x,\xi,E)=(\xi+i\varphi'_x(x))^2-E$ is the semiclassical leading symbol of $P_\varphi(h,E)$, and $\{f,g\}=\sum_{j}\p_{\xi_j}f\p_{x_j}g-\p_{x_j}f\p_{\xi_j}g$ is the Poisson bracket of the functions $f$ and $g$. 
Furthermore, we also assume that 
\begin{equation}
\label{eq_105_1}
|\nabla \varphi|> 0\quad \text{in}\quad \overline{\Omega}.
\end{equation}

We have the following local boundary Carleman estimates, see \cite{Lebeau_Robb_97},  \cite{Buffe_2017}. 
\begin{thm}
\label{thm_BLR_local}
Let $x_0\in \p \Omega$ and let $\omega$ be a sufficiently small neighborhood of  $x_0$ in   $\overline{\Omega}$. 
Let $\varphi$ satisfy  \eqref{eq_105_2}, \eqref{eq_105_1}, and 
\begin{equation}
\label{eq_105_3_new}
\p_\nu \varphi|_{\p \Omega\cap \overline{\omega}}>0. 
\end{equation}
Then there exist $0<h_0\le 1$ and  $C>0$ such that for all $0\le E\le 1$,  $k\ge 1$,    $0<h\le \frac{h_0}{k}$, and all $u\in H^2(\Omega)$ with $\supp(u)\subset \omega$ satisfying 
$(\p_\nu-ik)u=0$ on $\p\Omega\cap \omega$, we have 
\begin{equation}
\label{eq_thm_BLR_local}
\int_{\Omega} e^{\frac{2\varphi}{h}}|P(h,E) u|^2 dx \ge Ch \int_{\Omega} e^{\frac{2\varphi}{h}}(|u|^2+|h\nabla u|^2)dx.
\end{equation}
\end{thm}

\begin{proof}
Introducing boundary normal coordinates near the point $x_0$, we get a reduction to the case:
$\Omega=\R^n_+=\{x\in \R^n: x_n>0\}$, $\omega=\{x\in \overline{\R^n_+}:|x|<r_0\}$ for some $r_0>0$ small enough, $u\in H^2(\R^n_+)$, $\supp(u) \subset \omega$, and  $(\p_{x_n}-ik)u=0$ on $x_n=0$, see \cite{Lebeau_Robb_97}.   

It follows from \eqref{eq_105_3_new} that $\p_{x_n}\varphi>0$  in $\overline{\omega}$ for $r_0$ small enough.   Now thanks to  \cite[Proposition 2]{Lebeau_Robb_97}  the following Carleman estimate holds, see also  \cite[Theorem 1.5]{Buffe_2017}: there exist $C>0$, $h_1>0$ such that for all $0<h\le h_1$, $0\le E\le 1$, and all $u\in C^\infty(\overline{\R^n_+})$, $\supp(u)\subset \omega$, satisfying the inhomogeneous Neumann boundary conditions $\p_{x_n}u=g_0$ on $x_n=0$, we have
\begin{equation}
\label{eq_thm_BLR_local_new_05}
\begin{aligned}
\int_{\R^n_+} e^{\frac{2\varphi}{h}} |P(h,E) u|^2 dx+ h\int_{\R^{n-1}}e^{\frac{2\varphi}{h}}|hg_0|^2dx'\ge Ch  \int_{\R^n_+}e^{\frac{2\varphi}{h}}(|u|^2+|h\nabla u|^2)dx\\
+C h\int_{\R^{n-1}} e^{\frac{2\varphi}{h}}(|u(x',0)|^2+|h\nabla_{x'} u(x',0)|^2)dx'.
\end{aligned}
\end{equation}
We refer to \cite[Remark 3.8]{Burq_2002} for the explanation regarding  the uniformity of \eqref{eq_thm_BLR_local_new_05} in $0\le E\le 1$, see also \cite{Buffe_2017}. 

By density, \eqref{eq_thm_BLR_local_new_05} remains valid for   $u\in H^2(\R^n_+)$, $\supp(u)\subset \omega$, satisfying  $\p_{x_n}u=g_0$ on $x_n=0$. Applying now \eqref{eq_thm_BLR_local_new_05} to $u$ such that $\p_{x_n} u=iku$ on $x_n=0$, and letting $h_0>0$ be small enough so that the term $h\int_{\R^{n-1}}e^{\frac{2\varphi}{h}}(hk)^2 |u(x',0)|^2dx'$ can be absorbed into the right hand side of \eqref{eq_thm_BLR_local_new_05} for all $hk\le h_0$,  we get   
\[
\int_{\R^n_+} e^{\frac{2\varphi}{h}} |P(h,E) u|^2 dx \ge Ch  \int_{\R^n_+}e^{\frac{2\varphi}{h}}(|u|^2+|h\nabla u|^2)dx.
\]
This completes the proof of Theorem \ref{thm_BLR_local}. 
\end{proof}

We have the following corollary of Theorem \ref{thm_BLR_local}.
\begin{cor}
\label{cor_BLR_local}
Let $x_0\in \p \Omega$ and let $\omega$ be a sufficiently small neighborhood of  $x_0$ in   $\overline{\Omega}$. 
Let $\varphi$ satisfy  \eqref{eq_105_2}, \eqref{eq_105_1}, and \eqref{eq_105_3_new}. 
Then there exist $0<h_0\le 1$ and  $C>0$ such that for all $0\le E\le 1$,  $k\ge 1$,    $0<h\le \frac{h_0}{k}$, and all $v\in H^2(\Omega)$ with $\supp(v)\subset \omega$ satisfying 
$\p_\nu v -(\frac{\p_\nu \varphi}{h}+ik)v=0$ on $\p\Omega\cap \omega$, we have 
\[
\int_{\Omega} |P_\varphi (h,E) v|^2 dx \ge Ch \|v\|_{H^1_{\emph{\text{scl}}}(\Omega)}^2.
\]
Here $\|v\|_{H^1_{\emph{\text{scl}}}(\Omega)}^2=\|v\|_{L^2(\Omega)}^2+\|h\nabla v\|_{L^2(\Omega)}^2$.
\end{cor}

Next we shall state global boundary Carleman estimates with Robin boundary conditions by gluing  Theorem \ref{thm_Fursikov_Imanuvilov} and Corollary \ref{cor_BLR_local} together.  To that end, let  $\psi\in C^\infty(\overline{\Omega};\R)$ be such that $\psi(x)\ge 0$ for all $x\in \overline{\Omega}$ and $|\nabla \psi|>0$ in $\overline{\Omega}$.  Then  there is $\beta_0=\beta_0(\psi)>0$ sufficiently large such that 
$\varphi=e^{\beta_0 \psi}$ is a Carleman weight for the operator $P(h,E)=-h^2\Delta-E$, i.e. $\varphi$ satisfies  \eqref{eq_105_2}, uniformly in $0\le E\le 1$, see  \cite[Section 4.1]{Burq_2002}. We shall also assume that $\beta_0$ is so large that Theorem \ref{thm_Fursikov_Imanuvilov}  holds for $\varphi=e^{\beta_0 \psi}$. 
\begin{thm}
\label{thm_BLR}
Let $\emptyset\ne \Gamma\subset \p \Omega$ be open and let  $\psi\in C^\infty(\overline{\Omega};\R)$ be such that $\psi(x)\ge 0$ for all $x\in \overline{\Omega}$ and $|\nabla \psi|>0$ in $\overline{\Omega}$, and 
\begin{equation}
\label{eq_105_3}
\p_\nu\psi|_{\p \Omega\setminus\Gamma}>0. 
\end{equation}
Let $\varphi=e^{\beta_0 \psi}$, $\beta_0\gg 1$. Then there exist   $0<h_0\le 1$ and $C>0$ such that for all $0\le E\le 1$, $k\ge 1$, $0<h\le \frac{h_0}{k}$ and all $v\in H^2(\Omega)$ satisfying 
$\p_\nu v -(\frac{\p_\nu \varphi}{h}+ik)v=0$ on $\p \Omega$, we have 
\begin{equation}
\label{eq_thm_BLR}
\int_{\Omega} |P_\varphi(h,E) v|^2 dx+h\int_{\Gamma} (|v|_{\p \Omega}|^2+|h\nabla_\tau v|_{\p \Omega}|^2+|h\p_\nu v|_{\p \Omega}|^2)dS\ge Ch\|v\|^2_{H^1_{\emph{\text{scl}}}(\Omega)}.
\end{equation}
Here $\nabla_\tau$ is the tangential component of the gradient. 
\end{thm}

\begin{proof}
The assumption \eqref{eq_105_3} implies that there is an open set $\tilde \Gamma\subset \p \Omega$ so that $\tilde \Gamma\ne \p \Omega$,   $\p \Omega\setminus \Gamma\subset\subset \tilde \Gamma$, and 
\begin{equation}
\label{eq_107_1}
\p_\nu\psi|_{\tilde \Gamma}>0. 
\end{equation}
Let $\tilde \omega, \omega_1,\dots, \omega_M$ be a open cover of $\overline{\Omega}$ such that $\omega_1, \dots, \omega_M$ is an open cover of the boundary $\p \Omega$ so that $\omega_j$ are sufficiently small, and if $\omega_j\cap (\p \Omega\setminus \Gamma)\ne \emptyset$ then $\omega_j\cap\p \Omega\subset \tilde \Gamma$, $j=1,\dots, M$, and $\tilde \omega\cap \p \Omega=\emptyset$. 
Let $\tilde \chi\in C^\infty_0(\tilde \omega)$, $0\le \tilde \chi\le 1$, and $ \chi_j\in C^\infty_0(\omega_j)$, $0\le \chi_j\le 1$,  $j=1,\dots, M$, be such that $\tilde \chi +\sum_{j=1}^M\chi_j\ge 1$ near $\overline{\Omega}$. We can arrange so that $\p_\nu\chi_j|_{\p \Omega}=0$, $j=1,\dots, M$, see \cite{Imanuvilov_Uhlmann_Yamamoto_2015}. 

When $\omega_j\cap (\p\Omega\setminus \Gamma)=\emptyset$,  by Theorem \ref{thm_Fursikov_Imanuvilov}, we get for $h>0$ small enough, and $0\le E\le 1$, 
\begin{equation}
\label{eq_107_2}
\begin{aligned}
h\|\chi_j v\|_{H^1_{\text{scl}}(\Omega)}^2&\le C \|P_\varphi(h,E) v \|_{L^2(\Omega)}^2 +C\|[P_\varphi(h,E),\chi_j]v\|_{L^2(\Omega)}^2 \\
&+Ch\int_{\Gamma} ( |\chi_jv|_{\p \Omega}|^2+|h\nabla(\chi_j v)|_{\p \Omega}|^2)dS\\
&\le C \|P_\varphi(h,E) v\|_{L^2(\Omega)}^2 + \mathcal{O}(h^2)\|v\|_{H^1_{\text{scl}}(\Omega)}^2\\
& + Ch\int_{\Gamma} ( |\chi_jv|_{\p \Omega}|^2+|h \chi_j \nabla v|_{\p \Omega}|^2)dS+ \mathcal{O}(h^3) \int_{\Gamma}  |v|_{\p \Omega}|^2 dS.
\end{aligned}
\end{equation} 

When $\omega_j\cap(\p\Omega\setminus\Gamma)\ne \emptyset$, in view of \eqref{eq_107_1} and the fact that $\p_\nu (\chi_jv) -(\frac{\p_\nu \varphi}{h}+ik)(\chi_jv)=0$ on $\p \Omega$ , we can apply Corollary \ref{cor_BLR_local}, and obtain that  there exist  $0<h_0\le 1$ and $C>0$ such that for all  $0\le E\le 1$, $k\ge 1$, $0<h\le \frac{h_0}{k}$, 
\begin{equation}
\label{eq_107_3}
\begin{aligned}
h\|\chi_j v\|_{H^1_{\text{scl}}(\Omega)}^2&\le C \|P_\varphi(h,E) v \|_{L^2(\Omega)}^2 +C\|[P_\varphi(h,E),\chi_j]v\|_{L^2(\Omega)}^2 \\
&\le C \|P_\varphi(h,E) v \|_{L^2(\Omega)}^2 + \mathcal{O}(h^2)\|v\|_{H^1_{\text{scl}}(\Omega)}^2.
\end{aligned}
\end{equation} 
For the interior piece $\tilde \chi v$, the same estimate as \eqref{eq_107_3} holds. Summing up the estimates \eqref{eq_107_2} and \eqref{eq_107_3} and absorbing the error terms, we get \eqref{eq_thm_BLR}. This completes the proof. 
\end{proof}

We have the following corollary of Theorem \ref{thm_BLR}. 
\begin{cor}
\label{cor_BLR}
Let $\emptyset\ne \Gamma\subset \p \Omega$ be open and let  $\psi\in C^\infty(\overline{\Omega};\R)$ be such that $\psi(x)\ge 0$ for all $x\in \overline{\Omega}$ and $|\nabla \psi|>0$ in $\overline{\Omega}$, and \eqref{eq_105_3} holds. 
Let $\varphi=e^{\beta_0 \psi}$, $\beta_0\gg 1$. Then there exist   $0<h_0\le 1$ and $C>0$ such that for all $0\le E\le 1$, $k\ge 1$, $0<h\le \frac{h_0}{k}$ and all $u\in H^2(\Omega)$ satisfying 
$(\p_\nu -ik)u=0$ on $\p \Omega$, we have 
\begin{equation}
\label{eq_sec5_5}
\begin{aligned}
\int_{\Omega} e^{\frac{2\varphi}{h}} |P(h,E)u|^2 dx+h\int_{\Gamma} e^{\frac{2\varphi}{h}}(|u|_{\p \Omega}|^2&+|h\nabla_\tau u|_{\p \Omega}|^2+|h\p_\nu u|_{\p \Omega}|^2)dS\\
&\ge Ch\int_{\Omega} e^{\frac{2\varphi}{h}}(|u|^2+|h\nabla u|^2)dx. 
\end{aligned}
\end{equation}
\end{cor}

To use Corollary  \ref{cor_BLR}, we need  the following result on existence of a weight function with  special properties,   see \cite[Lemma 1.1]{Fursikov_Imanuvilov_1996},  \cite[Lemmas 2.1 and 2.3]{Imanuvilov_Yam_1998}.
\begin{thm}
\label{thm_weight}
Let $\emptyset\ne \Gamma\subset \p \Omega $ be an arbitrary open subset. Then there exists $\psi\in C^\infty(\overline{\Omega})$ such that 
\begin{align*}
&\psi(x)>0, \quad \forall x\in \Omega,\quad |\nabla \psi(x)|>0,\quad \forall x\in \overline{\Omega},\\
&\psi|_{\p\Omega\setminus \Gamma}=0, \quad    \p_\nu\psi|_{\p \Omega\setminus \Gamma}>0. 
\end{align*}
\end{thm}

\section{Consequences of Carleman estimates}

\label{sec_consequences}

Let $\Omega\subset \R^n$, $n\ge 3$, be a bounded domain with $C^\infty$ boundary.   Let $\omega_j \subset \Omega$  be neighborhoods of $\p \Omega$ with $C^\infty$ boundaries such that 
$\p \Omega\subset \p \omega_j$, $j=0,1,2,3$, and $\overline{\omega_{j}}\subset \omega_{j-1}$, $j=1,2,3$. Let $\emptyset\ne \Gamma\subset \p \Omega$ be an arbitrary non-empty open set.

Let  $q\in L^\infty(\Omega;\R)$, $k\ge 1$, and let $u\in H^2(\Omega)$ be such that 
\begin{equation}
\label{eq_sec5_1}
\begin{aligned}
&(-\Delta-k^2+q)u=0\quad \text{in}\quad \omega_0,\\
&(\p_\nu -ik)u=0 \quad \text{on}\quad \p \Omega.
\end{aligned}
\end{equation}

We have the following result in the case of Robin boundary conditions which is an analog of    \cite[Lemma 2.4]{Ben_Joud_2009}, obtained in the case of Dirichlet boundary conditions,  see also \cite{Bellass_Choulli_2009}. \begin{prop} 
\label{prop_based_on_Carleman_est}
There are constants $0<h_0\le 1$,  $C>0$, $\alpha_1>0$, and $\alpha_2>0$ such that for all $k\ge 1$, $0<h\le \frac{h_0}{k}$ and all $u\in H^2(\Omega)$ satisfying \eqref{eq_sec5_1}, we have 
\begin{equation}
\label{eq_based_on_Carleman_est}
\|u\|_{H^1(\omega_2\setminus\overline{\omega_3})}\le  C \big(e^{-\frac{\alpha_1}{h} } \|u\|_{H^1(\Omega)}+ e^{ \frac{\alpha_2}{ h}}\|u|_{\p \Omega}\|_{H^1(\Gamma)}\big).
\end{equation}
\end{prop}

\begin{proof}
Letting $h>0$, we rewrite  \eqref{eq_sec5_1} semiclassically as follows, 
\begin{equation}
\label{eq_sec5_2}
\begin{aligned}
&(-h^2\Delta- (hk)^2+h^2q)u=0\quad \text{in}\quad \omega_0,\\
&(\p_\nu -ik)u=0 \quad \text{on}\quad \p \Omega.
\end{aligned}
\end{equation}
Thanks to Theorem \ref{thm_weight} there exists $\psi\in C^\infty(\overline{\omega_0})$ such that 
\begin{equation}
\label{eq_sec5_100}
\begin{aligned}
\psi(x)>0, \quad \forall x\in \omega_0, \quad |\nabla \psi(x)|>0,\quad \forall x\in \overline{\omega_0},\\
\psi(x)=0,\quad \forall x\in \p \omega_0\setminus \Gamma,\quad \p_\nu\psi|_{\p \omega_0\setminus\Gamma}>0. 
\end{aligned}
\end{equation}
Let 
\begin{equation}
\label{eq_sec5_101}
\varphi=e^{\beta_0 \psi},
\end{equation}
 with $\beta_0>0$ sufficiently large as in Corollary \ref{cor_BLR}. 
 
 We shall now follow \cite{Ben_Joud_2009}, \cite{Bellass_Choulli_2009} closely. We observe that the fact $\psi(x)>0$, for all $x\in \omega_0$, implies that there exists $\kappa>0$ such that 
\begin{equation}
\label{eq_sec5_6}
\psi(x)\ge 2\kappa,\quad \forall x\in \omega_2\setminus\omega_3,
\end{equation}
and  the fact that  $\psi(x)=0$ for all  $x\in\p\omega_0\setminus \Gamma$ gives that there is a neighborhood $ \omega'$ of $\p \omega_0\setminus\p \Omega$ such that $ \omega'\cap\overline{\omega_1}=\emptyset$ and 
\begin{equation}
\label{eq_sec5_7}
\psi(x)\le \kappa, \quad \forall x\in \omega'. 
\end{equation}
Let $\omega''\subset \omega'$ be an arbitrary fixed neighborhood of $\p \omega_0\setminus\p \Omega$, and let $\theta\in C^\infty(\overline{\omega_0})$ such that $0\le \theta\le 1$,   $\theta=0$ on  $\omega''$ and $\theta=1$ on $\omega_0\setminus {\omega'}$. Setting $v=\theta u$, where $u$ satisfies \eqref{eq_sec5_2}, we get that $v$ satisfies the following problem,
\begin{equation}
\label{eq_sec5_102}
\begin{aligned}
&(-h^2\Delta-(hk)^2+h^2q)v=[-h^2\Delta, \theta]u\quad \text{in}\quad \omega_0,\\
&(\p_\nu -ik)v=0 \quad \text{on}\quad \p \omega_0.
\end{aligned}
\end{equation}
Applying the Carleman estimate \eqref{eq_sec5_5} for the operator $P(h,(kh)^2)=-h^2\Delta-(hk)^2$ on the domain $\omega_0$, with the Carleman weight $\varphi$ given by \eqref{eq_sec5_101}, \eqref{eq_sec5_100}, and  $v\in H^2(\omega_0)$ satisfying \eqref{eq_sec5_102}, we obtain that  there exist $0<h_0\le 1$ and $C>0$, such that for all $k\ge 1$,  $0<h\le \frac{h_0}{k}$,  
\begin{equation}
\label{eq_sec5_103}
\begin{aligned}
Ch\int_{\omega_0} e^{\frac{2\varphi}{h}}(|v|^2+|h\nabla v|^2)dx&\le 
\int_{\omega_0} e^{\frac{2\varphi}{h}} | P(h,(hk)^2)v |^2 dx\\
&+h\int_{\Gamma} e^{\frac{2\varphi}{h}}(|v|_{\p \omega_0}|^2+|h\nabla_\tau v|_{\p\omega_0}|^2)dS.
\end{aligned}
\end{equation}
Perturbing \eqref{eq_sec5_103} by $h^2q$ and using \eqref{eq_sec5_102}, we get 
\begin{equation}
\label{eq_sec5_104}
\begin{aligned}
Ch\int_{\omega_0} e^{\frac{2\varphi}{h}}(|v|^2+|h\nabla v|^2)dx&\le 
\int_{\omega_0} e^{\frac{2\varphi}{h}} | [-h^2\Delta, \theta]u |^2 dx\\
&+h\int_{\Gamma} e^{\frac{2\varphi}{h}}(|v|_{\p \omega_0}|^2+|h\nabla_\tau v|_{\p\omega_0}|^2)dS.
\end{aligned}
\end{equation}

As $\theta=1$ on $\omega_2\setminus \omega_3$ and in view of  \eqref{eq_sec5_6}, we get 
\begin{equation}
\label{eq_sec5_9}
Ch\int_{\omega_0} e^{\frac{2\varphi}{h}}(|v|^2+|h\nabla v|^2)dx\ge Ch e^{\frac{2}{h}e^{2\beta_0\kappa}}\|u\|^2_{H^1_{\text{scl}}(\omega_2\setminus\overline{\omega_3})}.
\end{equation}
Using that $ [-h^2\Delta, \theta]$ is a first order semiclassical differential operator such that  
\[
\supp ( [-h^2\Delta, \theta])\subset \omega'\setminus  \omega'',
\]
 and \eqref{eq_sec5_7}, we obtain that 
\begin{equation}
\label{eq_sec5_10}
\int_{\omega_0} e^{\frac{2\varphi}{h}} |  [-h^2\Delta, \theta]u |^2 dx\le e^{\frac{2}{h}e^{\beta_0\kappa}}h^2\|u\|^2_{H^1_{\text{scl}}( \omega'\setminus  \omega'')}.
\end{equation}
Finally, 
\begin{equation}
\label{eq_sec5_11}
h\int_{\Gamma} e^{\frac{2\varphi}{h}}(|v|_{\p \omega_0}|^2+|h\nabla_\tau v|_{\p \omega_0}|^2)dS\le h e^{\frac{2}{h}e^{\beta_0\|\psi\|_{L^\infty}}}\|u|_{\p \Omega}\|^2_{H^1_{\text{scl}}(\Gamma)}.
\end{equation}
Putting \eqref{eq_sec5_9}, \eqref{eq_sec5_10} and \eqref{eq_sec5_11} together, in view of \eqref{eq_sec5_104}, we have
\[
Ch e^{\frac{2}{h}e^{2\beta_0\kappa}}\|u\|^2_{H^1_{\text{scl}}(\omega_2\setminus\overline{\omega_3})}\le e^{\frac{2}{h}e^{\beta_0\kappa}}h^2\|u\|^2_{H^1_{\text{scl}}( \omega'\setminus  \omega'')}+ h e^{\frac{2}{h}e^{\beta_0\|\psi\|_{L^\infty}}}\|u|_{\p \Omega}\|^2_{H^1_{\text{scl}}(\Gamma)}.
\]
Setting 
\[
\alpha_1=e^{2\beta_0\kappa}-e^{\beta_0\kappa}>0, \quad \alpha_2=e^{\beta_0\|\psi\|_{L^\infty}}-e^{2\beta_0\kappa}>0,
\]
we get 
\[
\|u\|^2_{H^1_{\text{scl}}(\omega_2\setminus\overline{\omega_3})}\le  C \big(e^{-\frac{2\alpha_1}{h}}h \|u\|^2_{H^1_{\text{scl}}(\Omega)}+ e^{\frac{2\alpha_2}{h}}\|u|_{\p \Omega}\|^2_{H^1_{\text{scl}}(\Gamma)}\big),
\]
and therefore, 
\[
\|u\|_{H^1_{\text{scl}}(\omega_2\setminus\overline{\omega_3})}\le  C \big(e^{-\frac{\alpha_1}{h}}h^{\frac{1}{2}} \|u\|_{H^1(\Omega)}+ e^{\frac{\alpha_2}{h}}\|u|_{\p \Omega}\|_{H^1(\Gamma)}\big).
\]
Passing to the non-semiclassical $H^1$--norm and replacing $\alpha_1$, $\alpha_2$, by $\frac{\alpha_1}{2}$, $2\alpha_2$, respectively,  we obtain that 
\[
\|u\|_{H^1(\omega_2\setminus\overline{\omega_3})}\le  C \big(e^{-\frac{\alpha_1}{h}} \|u\|_{H^1(\Omega)}+ e^{\frac{\alpha_2}{h}}\|u|_{\p \Omega}\|_{H^1(\Gamma)}\big),
\]
for all $k\ge 1$,  $0<h\le \frac{h_0}{ k}$, and some $\alpha_1,\alpha_2>0$ independent of $h$ and $k$.  Thus,  the bound \eqref{eq_based_on_Carleman_est}
follows. This completes the proof of the proposition. 
\end{proof}

\section{Proof of Theorem \ref{thm_main}}
\label{sec_proof_thm_main}

We shall first follow the approach of \cite{Ben_Joud_2009}. Let $k\ge 1$ and let $u_2\in H^2(\Omega)$ be a solution to 
\begin{equation}
\label{eq_sec4_1}
(-\Delta-k^2+q_2)u_2=0\quad \text{in}\quad \Omega.
\end{equation}
In the sequel, we shall choose $u_2$ to be a complex geometric optics solution to \eqref{eq_sec4_1}.
Let $v\in H^1(\Omega)$ be the solution to the following problem
\begin{align*}
&(-\Delta-k^2+q_1)v=0\quad \text{in}\quad \Omega,\\
&(\p_\nu -ik)v=(\p_\nu -ik)u_2 \quad \text{on}\quad \p \Omega.
\end{align*} 
Then by the a priori estimate \eqref{eq_sec2_1_a_priori}, we conclude that $v\in H^2(\Omega)$.

Setting $u=v-u_2\in H^2(\Omega)$, we see that $u$ satisfies the following problem, 
\begin{equation}
\label{eq_sec4_1_-1}
\begin{aligned}
&(-\Delta-k^2+q_1)u=(q_2-q_1)u_2\quad \text{in}\quad \Omega,\\
&(\p_\nu -ik)u=0 \quad \text{on}\quad \p \Omega.
\end{aligned}
\end{equation}

Let us assume that $\omega_0\subset \Omega$ is a neighborhood of $\p \Omega$ with $C^\infty$ boundary where $q_1=q_2$. Now let $\omega_j\subset \Omega$ be neighborhoods of $\p \Omega$ with $C^\infty$ boundaries such that $\p \Omega\subset \p \omega_j$, $j=1,2,3$, and $\overline{\omega_j}\subset \omega_{j-1}$, $j=1,2,3$.  Let $\chi\in C^\infty_0(\Omega)$ be a cutoff function satisfying $0\le \chi\le 1$, $\chi=0$ on $\omega_3$ and $\chi=1$ on $\Omega\setminus \omega_2$. We set $\tilde u=\chi u$. Thus, we have
\begin{equation}
\label{eq_sec4_2}
(-\Delta-k^2+q_1)\tilde u=\chi(q_1-q_2)u_2+[-\Delta, \chi]u=(q_1-q_2)u_2+[-\Delta, \chi]u\quad \text{in}\quad \Omega.
\end{equation}

Let $u_1\in H^2(\Omega)$ be a solution to 
\begin{equation}
\label{eq_sec4_3}
(-\Delta-k^2+q_1) u_1=0\quad \text{in}\quad \Omega.
\end{equation}
Multiplying \eqref{eq_sec4_2} by $u_1$ and integrating by parts, we get 
\begin{equation}
\label{eq_sec4_4}
\int_\Omega (q_1-q_2)u_1u_2dx+\int_{\Omega}  [-\Delta, \chi]u u_1dx=0. 
\end{equation}
Now $[-\Delta, \chi]$ is a first order differential operator with $\supp( [-\Delta, \chi])\subset \overline{\omega_2}\setminus \omega_3$, and hence, we obtain that 
\begin{equation}
\label{eq_sec4_5}
\bigg| \int_{\Omega}  [-\Delta, \chi]u u_1dx \bigg|\le \|[-\Delta, \chi] u \|_{L^2(\omega_2\setminus \omega_3)} \|u_1 \|_{L^2(\Omega)}\le C\|u\|_{H^1(\omega_2\setminus\overline{\omega_3})}  \|u_1 \|_{L^2(\Omega)}. 
\end{equation}

Now it follows from \eqref{eq_sec4_4} and \eqref{eq_sec4_5} with the help of \eqref{eq_based_on_Carleman_est} that there are constants $0<h_0\le 1$,   $\alpha_1>0$, $\alpha_2>0$,  such that for $k\ge 1$, $0<h\le \frac{h_0}{k}$,  
\begin{equation}
\label{eq_sec4_6}
\bigg| \int_\Omega (q_1-q_2)u_1u_2dx\bigg| \le C\big( e^{-\frac{\alpha_1}{ h}} \|u\|_{H^1(\Omega)}  \|u_1 \|_{L^2(\Omega)}  + e^{ \frac{\alpha_2}{h} }\|u|_{\Gamma}\|_{H^1(\Gamma)}  \|u_1 \|_{L^2(\Omega)} \big).
\end{equation}

We have
\[
u|_{\Gamma}=v|_{\Gamma}-u_2|_{\Gamma}= \big(\Lambda_{q_1}^\Gamma(k)-\Lambda_{q_2}^\Gamma(k)\big)
\big( (\p_\nu-ik)u_2|_{\p \Omega}  \big).
\]
Using the mapping properties of the partial  Robin--to--Dirichlet map \eqref{eq_int_mapping_prop} and the trace theorem, we get 
\begin{equation}
\label{eq_sec4_7}
\begin{aligned}
\|u|_{\Gamma}\|_{H^1(\Gamma)}\le C\| \Lambda_{q_1}^\Gamma(k)-\Lambda_{q_2}^\Gamma(k) \|_{L^2(\p \Omega)\to H^1(\Gamma)}\| (\p_\nu-ik)u_2|_{\p \Omega} \|_{L^2(\p \Omega)}\\
\le  C\| \Lambda_{q_1}^\Gamma(k)-\Lambda_{q_2}^\Gamma(k) \|_{L^2(\p \Omega)\to H^1(\Gamma)} k\|u_2\|_{H^2(\Omega)}.
\end{aligned}
\end{equation}
Let us now bound $\|u\|_{H^1(\Omega)}$ in  \eqref{eq_sec4_6}. To that end, we recall that $u$ satisfies \eqref{eq_sec4_1_-1} and use Theorem \ref{thm_Baskin}. We get that there is $C>0$ such that for all $k\ge 1$,
\[
\|\nabla u\|_{L^2(\Omega)}+k\|u\|_{L^2(\Omega)}\le C(\|q_1\|_{L^\infty(\Omega)}\|u\|_{L^2(\Omega)}+\|q_1-q_2\|_{L^\infty(\Omega)}\|u_2\|_{L^2(\Omega)}), 
\]
and hence, for $k\ge 1$ sufficiently large, we obtain that  
 \[
\|\nabla u\|_{L^2(\Omega)}+\frac{k}{2}\|u\|_{L^2(\Omega)}\le C\|u_2\|_{L^2(\Omega)}.
\]
This implies in particular  that for $k\ge k_0\gg 1$, 
\begin{equation}
\label{eq_sec4_8}
\|u\|_{H^1(\Omega)}\le C\|u_2\|_{L^2(\Omega)}.
\end{equation}
By Proposition \ref{prop_uniform_bounded_k}, we have that \eqref{eq_sec4_8} is also valid for $k\in [1,k_0]$.

It follows from \eqref{eq_sec4_6},  \eqref{eq_sec4_7} and \eqref{eq_sec4_8} that for all $k\ge  1$,  $0<h\le \frac{h_0}{k}$, 
\begin{equation}
\label{eq_sec4_9}
\begin{aligned}
\bigg| \int_\Omega (q_1-q_2)u_1&u_2dx\bigg| \le C\big( e^{-\frac{\alpha_1}{ h}} \|u_2\|_{L^2(\Omega)}  \|u_1 \|_{L^2(\Omega)} \\
 &+ e^{ \frac{\alpha_2}{ h}}\| \Lambda_{q_1}^\Gamma(k)-\Lambda_{q_2}^\Gamma(k) \|_{L^2(\p \Omega)\to H^1(\Gamma)} \|u_2\|_{H^2(\Omega)} \|u_1 \|_{L^2(\Omega)} \big),
\end{aligned}
\end{equation}
for any $u_1,u_2\in H^2(\Omega)$ satisfying 
\begin{equation}
\label{eq_sec4_10}
(-\Delta-k^2+q_1)u_1=0, \quad (-\Delta-k^2+q_2)u_2=0,  \quad \text{in}\quad \Omega,
\end{equation}
respectively.  Here we used that $k\le \frac{1}{h}$ and replace $\alpha_2$ by $\alpha_2+1$.

Next let  $\tilde \Omega$ be open such that   $\Omega\subset\subset\tilde \Omega\subset\subset \R^n$, and let us extend $q_1$ and $q_2$ by zero to $\R^n\setminus \Omega$ and denote the extensions by $q_1$ and $q_2$ again. We shall take $u_1$ and $u_2$ to be complex geometric optics solutions constructed in Proposition \ref{prop_cgo_Hahner} on $\tilde \Omega$, and insert them into \eqref{eq_sec4_9}. To that end, let $\xi\in \R^n$ and $\mu_1, \mu_2\in \R^n$ be such that $|\mu_1|=|\mu_2|=1$ and $\mu_1\cdot\mu_2=\mu_1\cdot\xi=\mu_2\cdot\xi=0$. We set 
\[
\zeta_1=-\frac{\xi}{2}+\sqrt{k^2+a^2-\frac{|\xi|^2}{4}}\mu_1+ia \mu_2, \quad \zeta_2=-\frac{\xi}{2}-\sqrt{k^2+a^2-\frac{|\xi|^2}{4}}\mu_1-ia \mu_2,
\]
where $a$ is such that 
\begin{equation}
\label{eq_sec4_11_-1}
|\Im\zeta_j|=a\ge \max\{C_0M, 1\}
\end{equation}
 and $k^2+a^2\ge \frac{|\xi|^2}{4}$, see \cite{INUW_2014}.  Then we have $\zeta_j\cdot\zeta_j=k^2$, and by Propositions \ref{prop_cgo_Hahner} and \ref{prop_regularity_of_CGO}, there are solutions $u_j\in H^2(\Omega)$
to \eqref{eq_sec4_10} of the form
\begin{equation}
\label{eq_sec4_11}
u_j(x)=e^{i\zeta_j\cdot x}(1+r_j),
\end{equation}
where 
\begin{equation}
\label{eq_sec4_12}
\|r_j\|_{L^2(\tilde \Omega)}\le \frac{C_1}{a}\|q_j\|_{L^\infty(\Omega)},\quad j=1,2.
\end{equation}
Furthermore, thanks to Proposition \ref{prop_regularity_of_CGO}, for $k\ge 1$, we have
\begin{equation}
\label{eq_sec4_13}
\|u_j\|_{H^2(\Omega)}\le Ck^2\|u_j\|_{L^2(\tilde \Omega)}.
\end{equation}
In view of \eqref{eq_sec4_13}, \eqref{eq_sec4_9} with $u_j$ being geometric optics solutions \eqref{eq_sec4_11}, has the form  for all $k\ge 1$, $0<h\le \frac{h_0}{k}$,  
\begin{equation}
\label{eq_sec4_14}
\begin{aligned}
\bigg| \int_\Omega (q_1-q_2)u_1&u_2dx\bigg| \le C\big( e^{-\frac{\alpha_1}{ h}} \|u_2\|_{L^2(\Omega)}  \|u_1 \|_{L^2(\Omega)} \\
 &+ e^{\frac{ \alpha_2}{ h}}\| \Lambda_{q_1}^\Gamma(k)-\Lambda_{q_2}^\Gamma(k) \|_{L^2(\p \Omega)\to H^1(\Gamma)} \|u_2\|_{L^2(\tilde \Omega)} \|u_1 \|_{L^2(\Omega)} \big),
\end{aligned}
\end{equation}
where we have replaced $\alpha_2$ by $\alpha_2+1$, say. Let $R>0$ be such that $\tilde \Omega$ is contained in a ball centered at zero of radius $R$. Then 
thanks to \eqref{eq_sec4_11},  \eqref{eq_sec4_12}, and \eqref{eq_sec4_11_-1}, we have
\begin{equation}
\label{eq_sec4_15}
\|u_j\|_{L^2(\tilde \Omega)}\le Ce^{aR},\quad j=1,2.
\end{equation}
It follows from \eqref{eq_sec4_14} with the help of \eqref{eq_sec4_15} and  \eqref{eq_sec4_12} that 
\begin{equation}
\label{eq_sec4_16}
\begin{aligned}
\bigg| \int_\Omega (q_1-q_2) &e^{-i\xi\cdot x}dx\bigg| \\
&\le C\big( e^{-\frac{\alpha_1}{ h}} e^{2aR } + e^{ \frac{\alpha_2}{ h}} e^{2aR}\| \Lambda_{q_1}^\Gamma(k)-\Lambda_{q_2}^\Gamma(k) \|_{L^2(\p \Omega)\to H^1(\Gamma)} \big)+\frac{C}{a},
\end{aligned}
\end{equation}
for all $k\ge 1$, $0<h\le \frac{h_0}{k}$,  $\xi\in \R^n$ such that $|\xi|\le 2\sqrt{k^2+a^2}$ and $a\ge \max\{C_0M, 1\}$.

We shall take $\frac{1}{h}=\gamma a$ in \eqref{eq_sec4_16} and choose the constant $\gamma>0$ sufficiently large so that 
\begin{equation}
\label{eq_thm_main_0_7}
e^{-\alpha_1\gamma a+2a R}\le e^{-\alpha_3 a},\quad e^{\alpha_2 \gamma a +2aR}\le e^{\alpha_4 a},
\end{equation}
 for some constants $\alpha_3>0$ and $\alpha_4>0$. This implies that  $a\ge \frac{1}{h_0\gamma} k$.

Letting $\delta=\| \Lambda_{q_1}^\Gamma(k)-\Lambda_{q_2}^\Gamma(k) \|_{L^2(\p \Omega)\to H^1(\Gamma)} $, let us write  \eqref{eq_sec4_16} as
\begin{equation}
\label{eq_sec4_17}
|\mathcal{F}(q_1-q_2)(\xi)|\le  C\big( e^{-\alpha_3 a}  + e^{ \alpha_4 a} \delta +\frac{1}{a}\big) \le C\big( e^{ \alpha_4 a} \delta +\frac{1}{a}\big),
\end{equation}
for all $k\ge  1$,  $a\ge \frac{1}{h_0\gamma} k$,   $\xi\in \R^n$ such that $|\xi|\le 2\sqrt{k^2+a^2}$. 
Taking $\rho\le 2\sqrt{k^2+a^2}$ to be chosen and using \eqref{eq_sec4_17}, together with Parseval's formula, we get 
\begin{equation}
\label{eq_sec4_18}
\begin{aligned}
\|q_1-q_2\|_{H^{-1}(\Omega)}^2 \le \bigg(\int_{|\xi|\le \rho} +\int_{|\xi|\ge \rho}\bigg)\frac{|\mathcal{F}(q_1-q_2)(\xi)|^2}{1+|\xi|^2}d\xi \\
\le C\rho^n \bigg(   e^{2 \alpha_4 a}\delta^2 + \frac{1}{a^2} \bigg)+\frac{C}{\rho^2}.
\end{aligned}
\end{equation}
Setting $\rho=a^{\frac{2}{n+2}}$, \eqref{eq_sec4_18} gives that 
\begin{equation}
\label{eq_sec4_19}
\|q_1-q_2\|_{H^{-1}(\Omega)}^2\le C  \bigg(  a^{\frac{2n}{n+2}} e^{2 \alpha_4 a} \delta^2  + a^{-\frac{4}{n+2}}\bigg) \le C  \bigg(   e^{4 \alpha_4 a} \delta^2  + a^{-\frac{4}{n+2}}\bigg) ,
\end{equation}
for all $k\ge  1$,  and all $a\ge \frac{1}{h_0\gamma} k$. 
Using that $0<\delta<1/e$, and  choosing 
\[
a= \frac{1}{h_0\gamma} k+\frac{\log \frac{1}{\delta}}{4\alpha_4},
\]   
we conclude from \eqref{eq_sec4_19} that for all $k\ge  1$, 
\[
\|q_1-q_2\|^2_{H^{-1}(\Omega)}\le  e^{Ck}\delta+\frac{C}{(k+\log\frac{1}{\delta})^{\frac{4}{n+2}}}.
\]
This completes the proof of Theorem \ref{thm_main}. 

\textbf{Proof of Corollary \ref{cor_main_int}.} We follow the classical argument due to 
Alessandrini \cite{Aless_1}, see also \cite{Caro_Marinov}.   Let $\varepsilon>0$ be such that $s=\frac{n}{2}+2\varepsilon$. Then by the Sobolev embedding, interpolation and the a priori bounds for $q_j$, we get 
\begin{align*}
\|q_1-q_2\|_{L^\infty(\Omega)}&\le C\|q_1-q_2\|_{H^{\frac{n}{2}+\varepsilon}(\Omega)}\le C \|q_1-q_2\|_{H^{-1}(\Omega)}^{\frac{\varepsilon}{1+s}}\|q_1-q_2\|_{H^s(\Omega)}^{\frac{1-\varepsilon+s}{s+1}}\\
&\le C (2M)^{\frac{1-\varepsilon+s}{s+1}} \|q_1-q_2\|_{H^{-1}(\Omega)}^{\frac{\varepsilon}{1+s}}.
\end{align*}
Corollary \ref{cor_main_int} follows from this bound combined with Theorem \ref{thm_main}.

\begin{appendix}

\section{The interior impedance problem}
\label{sec_direct_problem} 

In this section we shall collect some well known results about the solvability of the interior impedance problem and some bounds on its solution needed in this paper, see \cite{BaSpWu_2016},  \cite{Ammari_Dos_Santos_2013}, and \cite{Melenk}. 

\begin{prop}
\label{prop_solvability_direct}
Let $\Omega\subset \R^n$, $n\ge 3$, be a bounded connected open set with $C^\infty$ boundary,  let $k>0$, and $q\in L^\infty(\Omega;\R)$. Then for any $F\in L^2(\Omega)$, $f\in H^{-\frac{1}{2}}(\p \Omega)$, the interior impedance problem, 
\begin{equation}
\label{eq_sec2_1}
\begin{aligned}
&(-\Delta-k^2+q)u=F\quad \text{in}\quad \Omega,\\
&(\p_\nu-i k)u=f\quad \text{on}\quad \p \Omega,
\end{aligned}
\end{equation}
has a unique solution $u\in H^1(\Omega)$.  Furthermore, there exists $C=C(k)>0$ such that 
\begin{equation}
\label{eq_sec2_1_cont_bound}
\|u\|_{H^1(\Omega)}\le C(\|F\|_{L^2(\Omega)}+\|f\|_{H^{-\frac{1}{2}}(\p \Omega)}).
\end{equation}
\end{prop}

\begin{proof}
Associated to \eqref{eq_sec2_1}, we introduce the following sesquilinear form
\begin{align*}
&a: H^1(\Omega)\times H^1(\Omega)\to \C,\\
&a(u,v)=\int_{\Omega} \nabla u\cdot \nabla \overline{v}dx +\int_{\Omega} (q-k^2) u\overline{v}dx +ik \int_{\p \Omega} u\overline{v}dS. 
\end{align*}
Now $u\in H^1(\Omega)$ is a solution to \eqref{eq_sec2_1} if and only if 
\[
a(u,v)=\int_{\Omega} F\overline{v}dx -\int_{\p \Omega} f\overline{v}dS,
\]
for all $v\in H^1(\Omega)$.  The form $a$ is bounded on $H^1(\Omega)\times H^1(\Omega)$, i.e. there is $C=C(k)>0$ such that
\[
|a(u,v)|\le C\|u\|_{H^1(\Omega)}\|v\|_{H^1(M)}, \quad u,v\in H^1(\Omega).
\]
and $a$ is coercive on $H^1(\Omega)$, i.e. 
\begin{equation}
\label{eq_sec2_1_-100}
\Re a(u,u)\ge \|u\|_{H^1(\Omega)}^2-C\|u\|_{L^2(\Omega)}^2, \quad u\in H^{1}(\Omega). 
\end{equation}

Let $\mathcal{A}: H^1(\Omega)\to (H^1(\Omega))^*$ be the bounded linear operator defined by the form $a$,
\[
a(u,v)=\langle \mathcal{A}u,\overline{v}\rangle_{(H^1(\Omega))^*, H^1(\Omega)}, \quad u,v\in H^1(\Omega). 
\]
Here $(H^1(\Omega))^*$ is the dual space to $H^1(\Omega)$. By Lax-Milgram's lemma and \eqref{eq_sec2_1_-100}, the operator $\mathcal{A}+C:H^1(\Omega)\to (H^1(\Omega))^*$ is an isomorphism provided that $C>0$ is sufficiently large, and since the imbedding  $H^1(\Omega) \hookrightarrow L^2(\Omega)$ is compact, we conclude that the operator $\mathcal{A}:H^1(\Omega)\to (H^1(\Omega))^*$ is Fredholm of index zero. 

Now let $u\in H^1(\Omega)$ be such that $\mathcal{A}u=0$. Then $u$ satisfies the impedance problem \eqref{eq_sec2_1} with $F=0$, $f=0$, and $\Im a(u,u)=0$. As $q$ is real valued, this implies that  $u=0$ and $\p_\nu u=0$ on $\p \Omega$. As $\Omega$ is connected,  by unique continuation for the equation $(-\Delta-k^2+q)u=0$ in $\Omega$, we get $u=0$ in $\Omega$. Thus, $\mathcal{A}$ is injective, and therefore, $\mathcal{A}:H^1(\Omega)\to (H^1(\Omega))^*$ is an isomorphism. 

By the trace theorem, for any $F\in L^2(\Omega)$ and $f\in H^{-\frac{1}{2}}(\p \Omega)$, the antilinear functional 
\[
H^1(\Omega)\ni v\mapsto \int_\Omega F\overline{v}dx-\int_{\p \Omega}f\overline{v}dS
\]
is continuous.  Hence, for any $F\in L^2(\Omega)$ and $f\in H^{-\frac{1}{2}}(\p \Omega)$, the problem \eqref{eq_sec2_1} has a unique solution $u\in H^1(\Omega)$ and \eqref{eq_sec2_1_cont_bound} holds. 
\end{proof}

We shall need the following mapping property of the Robin--to--Dirichlet map, introduced in \eqref{eq_int_Robin_to_Dir}. 
\begin{prop}
 \label{prop_regularity_RtD} 
 The operator $\Lambda_q(k): L^2(\p \Omega)\to H^{1}(\p \Omega)$ is bounded. 
\end{prop}
\begin{proof}
First it follows from \eqref{eq_sec2_1_cont_bound} that $\Lambda_q(k): H^{-\frac{1}{2}}(\p \Omega)\to H^{\frac{1}{2}}(\p \Omega)$ is bounded. The claim will follow by interpolation, if we show that $\Lambda_q(k): H^{\frac{1}{2}}(\p \Omega)\to H^{\frac{3}{2}}(\p \Omega)$ is bounded.  To see the latter, let $u\in H^1(\Omega)$ be the solution to \eqref{eq_int_1} with $f\in  H^{\frac{1}{2}}(\p \Omega)$. Then we have $\Delta u\in L^2(\Omega)$ and $\p_\nu u\in H^{\frac{1}{2}}(\p \Omega)$. By the a priori estimate 
\begin{equation}
\label{eq_sec2_1_a_priori}
\|u\|_{H^2(\Omega)}\le C(\|\Delta u\|_{L^2(\Omega)}+\|u\|_{H^1(\Omega)}+\|\p_\nu u\|_{H^{\frac{1}{2}}(\p \Omega)}),
\end{equation}
see \cite[Theorem 2.3.3.2, p. 106]{Grisvard_book} and  \cite[p. 253]{BaSpWu_2016}, we conclude that $u\in H^2(\Omega)$. Furthermore, combining the estimates \eqref{eq_sec2_1_a_priori} and \eqref{eq_sec2_1_cont_bound}, we see that $\Lambda_q(k): H^{\frac{1}{2}}(\p \Omega)\to H^{\frac{3}{2}}(\p \Omega)$ is bounded.  The result follows. 
\end{proof}

The following result will be needed when establishing Theorem  \ref{thm_main} for bounded frequencies. 
\begin{prop}
\label{prop_uniform_bounded_k}
Let $K\subset (0,\infty)$ be compact. There exists $C>0$ such that for all $k\in K$ and all $F\in L^2(\Omega)$, we have 
\[
\|u\|_{H^1(\Omega)}\le C \|F\|_{L^2(\Omega)}.
\]
Here $u\in H^1(\Omega)$ is the unique solution of 
\begin{equation}
\label{eq_sec2_1_with_f=0}
\begin{aligned}
&(-\Delta-k^2+q)u=F\quad \text{in}\quad \Omega,\\
&(\p_\nu-i k)u=0\quad \text{on}\quad \p \Omega. 
\end{aligned}
\end{equation}
\end{prop}

\begin{proof}
Assuming the contrary, we get sequences $k_n\in K$, $F_n\in L^2(\Omega)$ such that if $u_n\in H^1(\Omega)$ is the corresponding solution of \eqref{eq_sec2_1_with_f=0} then 
$\|u_n\|_{H^1(\Omega)}>n\|F_n\|_{L^2(\Omega)}$ for all $n=1,2,\dots$.  Assuming as we may that $\|u_n\|_{H^1(\Omega)}=1$, we get $\|F_n\|_{L^2(\Omega)}\to 0$ as $n\to \infty$ and  \eqref{eq_sec2_1_a_priori} implies that the sequence $u_n$ is bounded in $H^2(\Omega)$.  Using Rellich's compactness theorem,  we may assume, passing to subsequences, that $k_n\to k_0\in K$ and $u_n\to u_0$ in $H^1(\Omega)$, $\|u_0\|_{H^1(\Omega)}=1$. Using the weak formulation of the boundary problem \eqref{eq_sec2_1_with_f=0}, we obtain that 
\[
\int_{\Omega} \nabla u_n\cdot \nabla \overline{v} dx+\int_{\Omega}(q-k_n^2)u_n\overline{v}dx+ik_n\int_{\p \Omega} u_n\overline{v}dS=\int_{\Omega} F_n \overline{v} dx,  
\]
for all $v\in H^1(\Omega)$. Letting $n\to \infty$, we get $u_0=0$ which contradicts  the fact that $\|u_0\|_{H^1(\Omega)}=1$. 
\end{proof}

The following result, giving sharp bounds on solutions to the interior impedance problem, established recently by Baskin, Spence and Wunsch  \cite{BaSpWu_2016}, will be crucial for us when proving Theorem \ref{thm_main}. 
\begin{thm}
\label{thm_Baskin}
Let $\Omega\subset \R^n$, $n\ge 3$, be a bounded open set with $C^\infty$ boundary. Given  $F\in L^2(\Omega)$, $f\in L^2(\p \Omega)$, let $u\in H^1(\Omega)$ be the solution to the interior impedance problem, 
\begin{align*}
&(-\Delta-k^2)u=F\quad \text{in}\quad \Omega,\\
&(\p_\nu-i k)u=f\quad \text{on}\quad \p \Omega.
\end{align*}
Then there is $C>0$ such that 
\begin{equation}
\label{eq_Baskin}
\|\nabla u\|_{L^2(\Omega)}+|k| \|u\|_{L^2(\Omega)}\le C(\|F\|_{L^2(\Omega)}+\|f\|_{L^2(\p \Omega)}), 
\end{equation}
for all $k\in \R$. 
\end{thm}

\section{Proof of Theorem \ref{thm_Fursikov_Imanuvilov} }

\label{app_carleman}

We shall proceed by following the arguments of Fursikov and Imanuvilov \cite{Fursikov_Imanuvilov_1996} as presented in  \cite[Theorem 4.3.9]{Le_Rousseau_2012}.  By density, it is suffices to prove \eqref{eq_thm_Furs_Iman} for  $u\in C^\infty (\overline{\Omega})$. We write 
\[
P_\varphi(h,E)=A_2+iA_1,
\]
where 
\[
A_2=(hD)^2-|\varphi'|^2-E,\quad A_1=\varphi'\circ hD+hD\circ \varphi'=2\varphi'\cdot hD-i h\Delta \varphi.
\]
Here  $D=\frac{1}{i}\p$. The idea of Fursikov and Imanuvilov \cite{Fursikov_Imanuvilov_1996} is the following: rather than considering the equation $P_\varphi(h,E)u=g$, one works with 
\[
(A_2+i\underline{A_1})u=g+h\mu \Delta \varphi u,
\] 
where $\mu>0$ is to be chosen and   
\[
i\underline{A_1}=2\varphi'\cdot h\nabla + h(\mu+1)\Delta \varphi.
\]
Following  \cite{Fursikov_Imanuvilov_1996},  \cite[Theorem 4.3.9]{Le_Rousseau_2012},  we get
\begin{equation}
\label{eq_106_0}
\begin{aligned}
\| g+h\mu \Delta \varphi u\|_{L^2(\Omega)}^2&=\|A_2u\|_{L^2(\Omega)}^2+\|\underline{A_1} u\|^2_{L^2(\Omega)}+2\Re(A_2u, i\underline{A_1} u)_{L^2(\Omega)} \\
&\ge 2\Re(A_2u, i\underline{A_1} u)_{L^2(\Omega)}.
\end{aligned}
\end{equation}

We shall next compute 
\begin{equation}
\label{eq_106_1}
\Re(A_2u, i\underline{A_1} u)_{L^2(\Omega)}=\Re \int_{\Omega} ( (hD)^2u-|\varphi'|^2u-Eu ) (2\varphi'\cdot h\nabla \overline{u} + h(\mu+1)\Delta \varphi \overline{u})dx. 
\end{equation}
In doing so,  as in  \cite[Theorem 4.3.9]{Le_Rousseau_2012}, we write the integral in \eqref{eq_106_1} as a sum of six  terms
$I_{jk}$, $1\le j \le 3$, $1\le k\le 2$,  where $I_{jk}$ is the $L^2$ scalar product of the $j$th term in the expression of $A_2 u$ and the $k$th term  in the expression of $i\underline{A_1} u$.

For the term $I_{11}$ in \eqref{eq_106_1}, performing two integration by parts, as in \cite[Theorem 4.3.9]{Le_Rousseau_2012}, we get
\begin{align*}
I_{11}&=\Re \int_{\Omega} (-h^2 \Delta u) (2\varphi'\cdot h\nabla \overline{u}) dx=2h^3  \int_{\Omega} \varphi''\nabla u \cdot\nabla \overline{u}dx \\
&-h^3 \int_{\Omega}\Delta \varphi |\nabla u|^2dx 
-h^3\int_{\p \Omega} \p_\nu \varphi |\nabla u|^2dS
+2h^3\Re \int_{\p \Omega} (\p_\nu u)\varphi'\cdot \nabla \overline{u}dS.
\end{align*}
For the  term $I_{12}$ in \eqref{eq_106_1}, performing an integration by parts, as in \cite[Theorem 4.3.9]{Le_Rousseau_2012}, we obtain that 
\begin{align*}
I_{12}&=\Re \int_{\Omega} (-h^2 \Delta u) h(\mu+1)(\Delta \varphi) \overline{u}  dx= h^3(\mu+1)\int_{\Omega}\Delta \varphi |\nabla u|^2dx\\
 &+h^3(\mu+1)\Re \int_{\Omega} (\nabla u\cdot \nabla \Delta\varphi)\overline{u}dx+h^3(\mu+1)\Re\int_{\p \Omega} (\p_\nu u)(\Delta\varphi)\overline{u}dS.
\end{align*}
For the  term $I_{21}$ in \eqref{eq_106_1}, proceeding as in  \cite[Theorem 4.3.9]{Le_Rousseau_2012}, and performing an integration by parts,
we get 
\begin{align*}
I_{21}=-2\Re \int_{\Omega} |\varphi'|^2 u   \varphi'\cdot h\nabla \overline{u} dx= h\int_{\Omega} \nabla \cdot (|\varphi'|^2\varphi')|u|^2dx+h\int_{\p \Omega}\p_\nu \varphi |\varphi'|^2 |u|^2dS.
\end{align*}
For the term $I_{22}$ in \eqref{eq_106_1}, we have
\begin{align*}
I_{22}=-\Re \int_{\Omega} |\varphi'|^2 u   h(\mu+1)(\Delta \varphi) \overline{u}dx=-h(\mu+1)\int_{\Omega} (\Delta \varphi)|\varphi'|^2 |u|^2 dx.
\end{align*}
Finally, using that $\varphi'\cdot \nabla |u|^2=2\Re (u\varphi'\cdot \nabla \overline{u})$, and integrating by parts, we get 
\begin{align*}
I_{31}+I_{32}&=-\Re \int_{\Omega} Eu \big(  2\varphi'\cdot h\nabla \overline{u} + h(\mu+1)\Delta \varphi \overline{u}\big) dx=
-hE \int_{\Omega} \varphi'\cdot \nabla |u|^2dx\\
&-hE(\mu+1)\int_{\Omega} \Delta \varphi |u|^2dx= -hE \mu \int_{\Omega} \Delta \varphi |u|^2dx
+hE \int_{\p \Omega} (\p_\nu \varphi) |u|^2dS.
\end{align*}
Collecting all the terms together, we obtain that 
\begin{equation}
\label{eq_106_2_-1}
\begin{aligned}
\Re(A_2u, i\underline{A_1} u)_{L^2(\Omega)}=h\int_{\Omega}\tilde \alpha_0 |u|^2dx+h^3 \int_{\Omega} \alpha_1 |\nabla u|^2 dx +X + bt,
\end{aligned}
\end{equation}
where 
\begin{equation}
\label{eq_106_-2}
\begin{aligned}
&\tilde \alpha_0=\alpha_0-E\mu\Delta\varphi, \quad \alpha_0=\nabla \cdot (|\varphi'|^2\varphi') -(\mu+1)(\Delta \varphi)|\varphi'|^2, \quad \alpha_1=\mu\Delta\varphi,\\
&X=2h^3 \int_{\Omega} \varphi''\nabla u \cdot \nabla \overline{u}dx+ h^3(\mu+1)\Re \int_{\Omega}(\nabla \Delta \varphi \cdot \nabla u )\overline{u}dx,\\
&bt= -h^3\int_{\p \Omega} \p_\nu \varphi |\nabla u|^2dS
+2h^3\Re \int_{\p \Omega} (\p_\nu u)\varphi'\cdot \nabla \overline{u}dS\\
&+h^3(\mu+1)\Re\int_{\p \Omega} (\p_\nu u)(\Delta\varphi)\overline{u}dS+  h\int_{\p \Omega}\p_\nu \varphi |\varphi'|^2 |u|^2dS+
hE \int_{\p \Omega} (\p_\nu \varphi) |u|^2dS.
\end{aligned}
\end{equation}
Now by Lemma 4.3.10 in \cite{Le_Rousseau_2012}, we have
\begin{equation}
\label{eq_106_2}
\alpha_0\ge C\gamma^4\varphi^3,
\end{equation}
provided $\mu<2$. 
Assuming that $\gamma\ge 1$ and using that $\psi\ge 0$ on $\overline{\Omega}$,  we get  for all $0\le E\le 1$, 
\begin{equation}
\label{eq_106_3}
|E\mu\Delta\varphi|\le | \mu (\gamma^2 |\psi'|^2\varphi+\gamma \Delta \psi \varphi)|\le C \mu \gamma^3 \varphi^3. 
\end{equation}
 It follows from \eqref{eq_106_2} and \eqref{eq_106_3} that $\tilde \alpha_0\ge C \gamma^4 \varphi^3$ for $\gamma>1$ sufficiently large.  
As $\mu>0$,  we also have $\alpha_1\ge C\gamma^2\varphi$ for $\gamma$ sufficiently large. 
Hence, fixing $\mu=1$,  we conclude from \eqref{eq_106_2_-1}, by absorbing the remainder term $X$ as explained in  \cite{Le_Rousseau_2012}, that for all $h>0$ small enough,  all $\gamma$ large enough, and $0\le E\le 1$,
\begin{equation}
\label{eq_106_4}
\Re(A_2u, i\underline{A_1} u)_{L^2(\Omega)}\ge Ch \gamma^4 \int_{\Omega} \varphi^3 |u|^2 dx +Ch \gamma^2\int_{\Omega} \varphi |h\nabla u|^2dx -|bt|.
\end{equation}
It follows from \eqref{eq_106_-2} that for all $0\le E\le 1$, 
\begin{equation}
\label{eq_106_5}
|bt|\le C h \bigg (\gamma^3\int_{\p \Omega} \varphi^3 |u|^2dS + \gamma \int_{\p \Omega} \varphi |h\nabla u|^2dS\bigg). 
\end{equation}
Combining  \eqref{eq_106_0},  \eqref{eq_106_4} and \eqref{eq_106_5} and absorbing the term $h^2\mu^2 \|\Delta\varphi u\|^2_{L^2(\Omega)}$ by choosing $h$ small enough independent of $\gamma$, we get \eqref{eq_thm_Furs_Iman}. This completes the proof of Theorem \ref{thm_Fursikov_Imanuvilov}.

\section{Complex geometric optics solutions to Helmholtz equations}

\label{app_CGO}

Let $\Omega\subset \R^n$, $n\ge 2$, be a bounded open set, and let $k\ge 0$.  We recall the following result due to Sylvester--Uhlmann \cite{Syl_Uhl} and 
H\"ahner  \cite{Hahner_1996} concerning the existence of complex geometric optics solutions to  the Helmholtz equation,  see also \cite{FSU_book}.   This result is very useful here since all the constants are independent of the frequency $k$. 
\begin{prop}
\label{prop_cgo_Hahner}
Let $q\in L^\infty(\Omega)$ and $k\ge 0$. Then there are constants $C_0>0$ and $C_1>0$, depending on $\Omega$ and $n$ only, such that  for all $\zeta\in \C^n$,  $\zeta\cdot\zeta=k^2$, and $|\emph{\text{Im}} \zeta|\ge \max \{C_0\|q\|_{L^\infty(\Omega)}, 1\}$, the equation 
\[
(-\Delta-k^2+q)u=0\quad \text{in}\quad \Omega
\]
has a solution 
\[
u(x)=e^{i\zeta\cdot x}(1+r(x)),
\]
where $r\in L^2(\Omega)$ satisfies 
\[
\|r\|_{L^2(\Omega)}\le \frac{C_1}{|\emph{\text{Im}} \zeta|}\|q\|_{L^\infty(\Omega)}.
\]
\end{prop}

Let us recall the following standard  elliptic regularity result, see \cite[Theorem 7.1]{Zworski_book}
\begin{prop}
\label{prop_regularity_of_CGO}
Let $\Omega\subset\subset\tilde \Omega\subset \R^n$,  $q\in L^\infty(\tilde \Omega)$, and  $k\ge 0$. Let $u\in L^2(\tilde \Omega)$ be a solution to 
\[
(-\Delta-k^2+q)u=0\quad \text{in}\quad \tilde \Omega.
\]
Then $u\in H^2(\Omega)$ and  we have the following bounds,
\[
\|u\|_{H^2(\Omega)}\le C(1+k^2)\|u\|_{L^2(\tilde \Omega)}, \quad \|u\|_{H^1(\Omega)}\le C(1+k)\|u\|_{L^2(\tilde \Omega)}. 
\]
\end{prop}

\end{appendix}

\section*{Acknowledgements}
K.K. is very grateful to J\'er\^ome Le Rousseau and Luc Robbiano  for very helpful correspondence on Carleman estimates, and to J\'er\^ome Le Rousseau for bringing the reference \cite{Buffe_2017} to her attention. 
The research of K.K. is partially supported by the National Science Foundation (DMS 1500703, DMS 1815922). The research of G.U. is partially supported by NSF and a Si Yuan Professorship at IAS, HKUST. We are very grateful to the referees for their helpful comments.

\end{document}